\documentclass[reqno,10pt]{amsart}
\usepackage[notref,notcite]{showkeys}
\usepackage{hyperref}
\begin{document}
\title[On the asymptotic behaviour of a fifth order difference equation]{On the asymptotic behaviour of the solutions  of a fifth order difference equation}
\author[G. L. Karakostas]
{George L. Karakostas}

\address{George L. Karakostas \newline
 Department of Mathematics, University of Ioannina,
 451 10 Ioannina, Greece}
\email{gkarako@uoi.gr, gkarako@hotmail.com}

\date{June 2024}
 \begin{abstract} The first aim of this note is to make clear what is the equilibrium of a  fifth order difference equation studied in the literature. Next  the  investigation of the whole asymptotic behaviour of the solutions of the equation is presented.   \end{abstract}
 \maketitle
\numberwithin{equation}{section}
\newtheorem{theorem}{Theorem}[section]
\newtheorem{lemma}[theorem]{Lemma}
\newtheorem{definition}[theorem]{Definition}
\newtheorem{example}[theorem]{Example}
\newtheorem{property}[theorem]{Property}
\newtheorem{remark}[theorem]{Remark}
\section{Introduction}
In this note we deal with  the difference equation 
\begin{equation}
\label{a}x_{n+1}= ax_{n-1}+ \frac{bx_{n-1}x_{n-4}}{cx_{n-4} + dx_{n-2}},
\end{equation}
where the parameters $a, b, c, d$ and the initial values $x_j,$ $ j = 0, -1, -2, -3, -4$ are positive real numbers.  
Equation \eqref{a} has been recently
investigated by Elsayed, and his coauthors in \cite{EAM} where  in some special
cases, closed-form formulas for the solutions  are given, and  some facts on the boundedness and  stability  are exhibited. More precisely,  it is claimed that under the condition  $(1-a)(c+d)\neq b,$ the point 0 is an equilibrium and a  global attractor. Then Stevi\'c, et all in  \cite{SIK} (pp. 22668, 22669) express their doubts about the correctness  of the results in \cite{EAM} and they claim that 0 can not be an equilibrium, because  in the process of solving equation $$ax+\frac{bx^2}{(c+d)x}=x,$$ the point 0 must be excluded. Also, they say that it can happen that some solutions may converge to something which is not an equilibrium.   And in order to  support their claim, they give an example which  shows that  unbounded solutions exist. \par 
With motivation the previous remarks, in this note  we show that, indeed, the results in \cite{EAM} are not correct, and the behaviour of the solutions depends on the parameters.  In particular, we show that if $A:=(c+d)(1-a)-b>0,$ the point  0 is an equilibrium which attracts all  solutions with positive initial values, while, if $A<0$, any such solution converges to $+\infty.$ The case $A=0$ is more interesting and we shall discuss  it in the text.
  \section{The main results}  We begin with a little elaboration on the definition of the equilibrium, considering a general difference equation \begin{equation}\label{e2}x_{n+1}=G(x_n, x_{n-1}, \cdots, x_{n-k}).\end{equation}  Assume that a solution $(x_n)$ of equation \eqref{e2} converges to some finite point $a$, as $n$ tends to $+\infty$.   Then, it holds $$\lim_{n\to+\infty} x_{n-m}=a,$$ for any fixed $m.$ Passing to the limits, from \eqref{e2} we get $a=G(a, a, \cdots,a),$ whenever the right argument is defined. This  means that $a$ is a fixed point of the equation.  Of course this happens when $G$ is a continuous function at $a$. Hence,  in order to find the equilibria of equation \eqref{e2} we must solve the equation 
$$a=\lim_{v\to a}G(v, v, \cdots, v),$$ whenever the limit in the right side exists. 
\par Our purpose here is to discuss the existence and stability of the equilibrium of  \eqref{a}, where in finding the fixed points  the factor $0/0$ appears. Furthermore, we discuss the asymptotic behaviour of the solutions of equation \eqref{a} and we show that this behaviour depends on the sign of the quantity  $A$ defined above and of $B:=bd-(c+d)^2.$ 
\par  Consider  the algebraic equation \eqref{a}. Solving equation   $$au+\lim_{v\to u}\frac{bv^2}{(c+d)v}=u$$ we obtain $u=0$, which is the unique fixed point,  provided that the condition $A\neq 0$ is satisfied.\par As we said earlier, equation \eqref{a} is studied in \cite{EAM}, where the results seem to be not correct. Here we discuss the asymptotic behaviour of the solutions of \eqref{a}, but, first, we need to refer the following result, and notice that,   it has been proved as an important tool in the literature: 
 \begin{theorem}\label{T1} Theorem 2 (\cite{CFR}, p. 311) If $T$ is an $n$-dimensional 
differentiable function with fixed point $X$ and $J$ is the Jacobian matrix of $T$ evaluated at $X$, then $X$ is a locally stable fixed point if all eigenvalues of $J$ have absolute value less than 1. If at least one of these absolute values is strictly greater than 1, the  fixed point is unstable.\end{theorem}
Our main result concerning \eqref{a} is given in the following theorem:
\begin{theorem} Assume that the parameters $a, b, c, d$ are as above. Then we have the following facts:\par
1) If $A>0$, then all positive solutions converge  to $0$. \par
2)  If  $A<0$,  then all positive solutions converge  to $+\infty$. 
\par 3) If $A=0,$  then, for any positive real number $w$, the sequence  $x_n=w, \enskip n=-4, -3, -2, -1,  \cdots$ is a solution, which, \par \hskip .2in 3a) if $B>0$, then it is unstable. Moreover, \par \hskip.2in 3b)  if  $B\leq 0$, then there exist solutions which converge to a positive real number $\neq 0.$     \end{theorem}
\begin{proof} Let $(x_n)$ be a solution of \eqref{a} with positive initial values. Obviously, it stays positive for all indices $n.$ We set $$y_{n+1}:=\frac{x_{n+1}}{x_{n-1}}$$ and next let $$w_n:=c+dy_{3n}.$$ Then from  the original equation we obtain  $$w_{n+1}=c+ad+\frac{bd}{w_n}.$$ For the sequence $u_0:=1$, $u_n:=w_1w_2\cdots w_n,$ we obtain $w_n=\frac{u_n}{u_{n-1}}$ and the relation
$$u_{n+1}-(c+ad)u_n-bdu_{n-1}=0,$$ which is a linear difference equation. Its characteristic equation has the roots
$$\rho_{\pm}=\frac{1}{2}\big[c+ad\pm((c+ad)^2+4bd)^{1/2}\big]$$ and so we get
$$w_{n+1}=\frac{l_1\rho_+^{n+1}+l_2\rho_-^{n+1}}{l_1\rho_+^{n}+l_2\rho_-^{n}},$$ for some  constants $l_1, l_2,$ which depend on the initial values of the solution. 
Since $\rho_+>\rho_-,$ we obtain $\lim w_n=\rho_+.$ If we set $w_n':=c+dy_{3n+1}$ or $w''_n:=c+dy_{3n+2}$, then we obtain the same result. These facts imply that
$$\lim\frac{x_{n+1}}{x_{n-1}}=\frac{\rho_+-c}{d}.$$ \par Now, we distinguish the following  cases:\par Case 1) $\frac{\rho_+-c}{d}>1$, which holds if and only if  $A=(c+d)(1-a)-b<0$. Choose any $\kappa$ such that $\frac{\rho_+-c}{d}>\kappa>1.$ Then there is some $n_0$ such that 
$$\frac{x_{n+1}}{x_{n-1}}>\kappa, \enskip n\geq n_0.$$ Thus we have
$$x_{n+1}>\kappa x_{n-1}>\kappa^2 x_{n-3}>\cdots>\kappa^{m_n} x_{j_0},$$ for some $j_0\in\{n_0, n_0+1\}$ and some sequence $m_n\to +\infty.$ In this case the solution converges to $+\infty$\footnote{Notice that, if we get $a=c=\frac{1}{2}$ and $b=d=1$, then we obtain that $\frac{\rho_+-c}{d}=\frac{\sqrt{5}}{2}>1,$ as it was found in \cite{SIK}.}. \par 
Case 2) $\frac{\rho_+-c}{d}<1$, which holds if and only if $A>0$. Then, working as above, we can  see that the solution tends to zero. \par
Case 3) $\frac{\rho_+-c}{d}=1$. In this case we have $A=(c+d)(1-a)-b=0$ and equation \eqref{a} becomes 
\begin{equation}\label{eq} x_{n+1}=(1-\frac{b}{c+d})x_{n-1}+\frac{bx_{n-1}x_{n-4}}{cx_{n-4}+dx_{n-2}}.\end{equation}
\par To continue, we rewrite equation \eqref{eq} as a system in the 5-th dimensional space, as follows: We let $$y^{(1)}_n:=x_{n-4},\enskip y^{(2)}_n:=x_{n-3}, \enskip y^{(3)}_n:=x_{n-2}, \enskip y^{(4)}_n:=x_{n-1}, \enskip y^{(5)}_n:=x_{n}.$$Then the vector $Y_n:=(y^{(1)}_n, y^{(2)}_n, y^{(3)}_n, y^{(4)}_n, y^{(5)}_n)$ satisfies the equation
\begin{displaymath}
{Y_{n+1}} =
\left( \begin{array}{c}
y^{(2)}_n \\
y^{(3)}_n\\
y^{(4)}_n\\
y^{(5)}_n\\
(1-\frac{b}{c+d})y^{(4)}_n+\frac{by^{(4)}_ny^{(1)}_n}{cy^{(1)}_n+dy^{(3)}_n}
\end{array} \right).
\end{displaymath} \par 
It is clear that a fixed point of this equation is a vector of the form ${\bf{w}}:=(w, w, w, w, w)^T,$ where $w$ is any positive real number. \par A linearisation of this equation at a fixed vector ${\bf{w}}$ gives the equation
$$Y_{n+1}=AY_n,$$ where
\begin{displaymath}
A =
\left( \begin{array}{ccccc}
0&1&0&0&0 \\
0&0&1&0&0\\
0&0&0&1&0\\
0&0&0&0&1\\
p&0&-p&1&0
\end{array} \right), 
\end{displaymath} where 
$p:=\frac{bd}{(c+d)^2}.$\par 
Obviously the matrix $A$ does not depend on the vector $\bf{w}.$ The eigenvalues of the matrix $A$ are the real numbers $\pm1,$ $ p^{1/3},$ $ p^{1/3}(1\pm i\sqrt{3})/2.$ Hence its spectral radius $\rho$ is equal to $\max\{1, p^{1/3}\}.$ 
Therefore, according to  theorem \ref{T1}, in case $B>1$ we have $ \rho=p^{1/3}>1,$ and so any fixed point $\bf{w}$ is unstable. This proves the case (3a). 
\par
In case (3b), we have $B\leq 0$ and so $\rho=1$. In order to see what can happen in this case, we give the following application: Consider the solution $(x_n)$ with initial values  \begin{equation}\label{v1}x_0=x_1=x_2=1\enskip\text{and}\enskip x_3=x_4=\mu,\end{equation} where $\mu$ is any positive real number. Obviously, if $\mu=1$ the solution is the constant sequence $x_n=1,$ $n=0, 1, \cdots.$ So, we assume that $\mu\neq 1.$ \par \par First, we shall show that the solution is given by $$x_{3n+j}=\prod_{i=0}^{n}K_i,\enskip \text{for any}\enskip j=0, 1, 2\enskip\text{and}\enskip n=0, 1, 2, \cdots$$ where 
\begin{equation}\label{m}K_0:=1, \enskip K_1:=\mu, \enskip K_{i+1}:=\frac{1}{2}+\frac{1}{1+K_{i}}, \enskip i=1, 2, \cdots.\end{equation} Indeed, we observe that 
$$x_{3n}=x_{3n-2}\big(\frac{1}{2}+\frac{x_{3n-5}}{x_{3n-5}+x_{3n-3}}\big)=\prod_{i=0}^{n-1}K_i\big(\frac{1}{2}+\frac{\prod_{i=0}^{n-2}K_i}{\prod_{i=0}^{n-2}K_i+\prod_{i=0}^{n-1}K_i}\big)=$$
$$=\prod_{i=0}^{n-1}K_i\big(\frac{1}{2}+\frac{1}{1+K_{n-1}}\big)=\prod_{i=0}^{n}K_i.$$ Exactly the same things we have for the value $x_{3n+1}$, while for $x_{3n+2}$ we get
$$x_{3n+2}=x_{3n}\big(\frac{1}{2}+\frac{x_{3n-3}}{x_{3n-3}+x_{3n-1}}\big)=\prod_{i=0}^{n}K_i\big(\frac{1}{2}+\frac{\prod_{i=0}^{n-1}K_i}{\prod_{i=0}^{n-1}K_i+\prod_{i=0}^{n-1}K_i}\big)=$$
$$=\prod_{i=0}^{n}K_i\big(\frac{1}{2}+\frac{1}{1+1}\big)=\prod_{i=0}^{n}K_i.$$ \par Next we claim  that the sequence $(K_i)$ converges to 1. \par Indeed, we observe that for each $i$ we have $K_i>0,$ thus $K_{i+1}>\frac{1}{2},$ $i=0, 1, 2, \cdots.$ Hence $K_i<\frac{1}{2}+\frac{1}{1+\frac{1}{2}}=\frac{7}{6}.$ This means that the sequence $(K_i)$ is bounded and it stays away from zero. Therefore there are two full limiting sequences (see, e.g., \cite{K1, K2, K3}) $\{M_i:\enskip i\in\mathbb{Z}\}$ and $\{N_i:\enskip i\in\mathbb{Z}\}$  such that $$\frac{1}{2}\leq N_0=\liminf K_i\leq\limsup K_i=M_0\leq\frac{7}{6}.$$ Then from the relations 
$$M_0=\frac{1}{2}+\frac{1}{1+M_{-1}}\leq\frac{1}{2}+\frac{1}{1+N_0}=\frac{3+N_0}{2(1+N_0)}$$ and 
$$N_0=\frac{1}{2}+\frac{1}{1+N_{-1}}\geq\frac{1}{2}+\frac{1}{1+M_0}=\frac{3+M_0}{2(1+M_0)}$$ we obtain $M_0=N_0,$ which means that $$\limsup K_i=\liminf K_i=:P\in[\frac{1}{2}, \frac{7}{6}].$$ Then from  \eqref{m} we get the relation $P=\frac{1}{2}+\frac{1}{1+P}$ which is satisfied when $P=1,$ which proves our claim. \par We are interesting in the limit of  the solution $(x_n)$. First we discuss the existence of the limit of the sequence of the products $\prod_{i=0}^{n}K_i$, or, equivalently, the "sum" of the series $\sum_{i=0}^{+\infty}\ln K_i.$ 
We shall show that the latter is an alternating series and it satisfies the Leibniz rule. \par Indeed, as we proved, it holds $\lim_i\ln K_i=\ln 1=0.$ \par To proceed, we shall discuss the cases $\mu<1$ and $\mu>1.$  \par First, consider the case $\mu\in(0,1),$  thus we have  $\ln K_1=\ln \mu<0$ and $\ln K_2=\ln \frac{3+\mu}{2(1+\mu)}>\ln 1=0.$ To apply induction, assume that  for some $i$ it holds $\ln K_{2i-1}<0.$ Then we have $K_{2i-1}<1$ and so $$\ln K_{2i}=\ln(\frac{1}{2}+\frac{1}{1+K_{2i-1}})>\ln(\frac{1}{2}+\frac{1}{1+1})=0.$$ This means that $K_{2i}>1$ and so, again, 
$$\ln K_{2i+1}=\ln(\frac{1}{2}+\frac{1}{1+K_{2i}})<\ln(\frac{1}{2}+\frac{1}{1+1})=0.$$ These facts show that the series $\sum_{i=0}^{+\infty}\ln K_i,$ is alternating. It remains to show that the absolute values of the terms form a decreasing sequence. This means that we have to show that $\ln K_{2i}>-\ln K_{2i+1}$ and $-\ln K_{2i+1}>\ln K_{2i+2}.$ The first one is equivalent to $$1<K_{2i}K_{2i+1}=K_{2i}(\frac{1}{2}+\frac{1}{1+K_{2i}}),$$ or to $K_{2i}>\frac{-1+3}{2}=1,$ which is true. Similarly, we can see that the relation $-\ln K_{2i+1}>\ln K_{2i+2}$ holds, if and only, if $K_{2i+1}<1,$ which, also, is true. \par
\par Next assume that $\mu>1.$ Then we have  $\ln K_1=\ln \mu>0$ and $\ln K_2=\ln \frac{3+\mu}{2(1+\mu)}<\ln 1=0.$  As previously, we can see that $\ln K_{2i}<0$ and $\ln K_{2i+1}>0,$ for all $i,$ and moreover, the sequence of the absolute values is decreasing. \par Therefore by the Leibniz rule we conclude that the sum $\sum_{i=0}^{+\infty}\ln K_i$ exists in $\mathbb{R}$, or equivalently, the limit $\lim_{n\to+\infty}\prod_{i=0}^{n}K_i=\lim x_{3n+j}$ exists in $(0,+\infty)$ and it is independent of $j=0, 1, 2.$ This means that the sequence $(x_n)$ converges to a certain positive real number. \par To get an idea about the location of this limit for the value $\mu=\frac{1}{2}$, one can see that  the first values of $\prod_{i=0}^{n}K_i,$ for $n=1, 2, \cdots, 9, 10$ form the sequence 0.5, 0.58333, 0.56089, 0.56639, 0.56501, 0.56535, 0.56527, 0.56529, 0.56528, 0.56528.  It seems that for all $n>9$ we have the same value 0.56528, which shows that this is an approximate limit of the solution $(x_n)$ of equation \eqref{eq} with the given values \eqref{v1}.\end{proof}
\par
 \end{document}